\documentclass[12pt]{amsart}
\usepackage{amsmath,latexsym,amsfonts,amssymb,amsthm}
\usepackage{geometry}
\usepackage{hyperref}
\usepackage{mathrsfs}
\usepackage{mathtools}
\usepackage[all,cmtip]{xy}
\usepackage[alphabetic]{amsrefs}

\newtheorem{prop}{Proposition}
\newtheorem{thm}[prop]{Theorem}
\newtheorem{cor}[prop]{Corollary}

\newtheorem{lem}[prop]{Lemma}

\theoremstyle{definition}
\newtheorem{que}[prop]{Question}

\newtheorem{expl}[prop]{Example}
\newtheorem{rem}[prop]{\it Remark}

\newtheorem*{claim*}{Claim}
\newtheorem{say}[prop]{}

\newcommand{\bP}{\mathbb{P}}
\newcommand{\bC}{\mathbb{C}}

\newcommand{\bZ}{\mathbb{Z}}
\newcommand{\bN}{\mathbb{N}}

\newcommand{\bG}{\mathbb{G}}

\newcommand{\cO}{\mathcal{O}}

\DeclareMathOperator{\Supp}{Supp}

\DeclareMathOperator{\rank}{rank}

\DeclareMathOperator{\ed}{ed}

\DeclareMathOperator{\tf}{tf}

\numberwithin{equation}{section}

\renewcommand{\o}[0]{{\mathcal O}} 
\newcommand{\z}[0]{{\mathbb Z}}

\newcommand{\p}[0]{{\mathbb P}}

\newcommand{\spec}[0]{\operatorname{Spec}}

\newcommand{\chr}[0]{\operatorname{char}}
\newcommand{\GL}{\mathrm{GL}}
\newcommand{\tsum}[0]{\textstyle{\sum}} 
\newcommand{\onto}[0]{\twoheadrightarrow}
\newcommand{\qtq}[1]{\quad\mbox{#1}\quad}
\newcommand{\alb}[0]{\operatorname{Alb}}

\title{Essential dimension of isogenies}

\author{J\'anos Koll\'ar}

\address{Department of Mathematics, Princeton University, Princeton, NJ 08544, USA}
\email{kollar@math.princeton.edu}

\author{Ziquan Zhuang}

\address{Department of Mathematics, Johns Hopkins University, Baltimore, MD 21218, USA}
\email{zzhuang@jhu.edu}

\date{}

\dedicatory{Dedicated to James M\textsuperscript{c}\!Kernan on the occasion of his sixtieth birthday.}

\begin{document}

\begin{abstract}
  We give a lower bound for the essential dimension of isogenies of complex abelian varieties. The bound is sharp in many cases. In particular, 
 the multiplication-by-$m$ map  is incompressible  for every $m\geq 2$,  confirming a conjecture of Brosnan.
\end{abstract}

\maketitle

\section{Introduction}

The {\it essential dimension} of a generically finite map $f\colon Y\to X$ is the smallest integer $m$ such that  $f$ is birational  to the pull-back  of a map of varieties of dimension $m$.
It is denoted by  $\ed(Y/X)$ or  $\ed(f)$. 
A map 
$f$ is called {\it incompressible} if
its  essential dimension is $\dim X$.

This notion was introduced in \cite{BR-ess-dim} and has been studied in a variety of contexts; see e.g. \cites{Rei-ess-dim-survey,Mer-ess-dim-survey} for surveys and \cite{FW-resolvent-degree} for its connection to an algebraic form of Hilbert's 13th problem through a variant call the resolvent degree. 

The purpose of this note is to describe a birational geometric approach to this invariant, leading to a proof of a conjecture of  Brosnan, recorded in  \cite{FS-incompressibility}*{Conjecture 6.1}.
Throughout we work over an  algebraically closed field of characteristic $0$.

\begin{thm} \label{cor:mult-by-p incompressible}
Let $A$ be an abelian variety. Then  the multiplication-by-$m$ map $m_A: A\to A$ is incompressible for  every $m\geq 2$. 
\end{thm}

The first lower bounds on the   essential dimension of $m_A$ are in
\cite{FS-incompressibility}*{Theorem 6.4}.
Incompressibility is proved in  \cite{FKW-ess-dim-prismatic}, provided
 $m$ has a sufficiently large prime factor (depending on $A$). 
The proof uses a 
prismatic cohomology argument and also proves a stronger result, called
$p$-incompressibility.

%% also establishes a few other special cases of Brosnan's conjecture.

Our main technical result is a lower bound for the essential dimension
of arbitrary isogenies. For an abelian group $G$ we let $\rank( G)$ denote the minimum number of generators, and set $\rank_p (G)=\rank (G/pG)$. 

\begin{thm} \label{thm:ess.dim.}
Let $\alpha\colon A\to A'$ be an isogeny of abelian varieties. Then there exists an abelian subvariety $B$ of $A$ such that, for every prime $p$,
\[
\ed(\alpha)\ge \dim A -\dim B +  \tfrac{p-1}{p}\rank_p(\ker(\alpha)\cap B).
\eqno{(\ref{thm:ess.dim.}.1)}
\]
\end{thm}

If $\alpha=m_A$ is  multiplication by $m$ and $p\mid m$, then
$\rank_p(\ker(\alpha)\cap B)=2\dim B$, so the right hand side of
(\ref{thm:ess.dim.}.1) is $\geq \dim A$, proving Theorem~\ref{cor:mult-by-p incompressible}.

We also  get a formula for the essential dimension of isogenies whose degree has no  small prime factors.

\begin{cor} \label{cor:formula}
Let $\alpha\colon A\to A'$ be an isogeny of abelian varieties whose degree is coprime to $(\dim A)!$. Then the essential dimension of $\alpha$ equals the minimum of 
\[
\dim A -\dim B + \rank(\ker(\alpha)\cap B),
\]
as $B$ runs through all abelian subvarieties of $A$.

%% \footnote{The rank of a finite abelian group $G$ is the smallest number of elements needed to generate $G$.}as $B$ varies among the abelian subvarieties of $A$.}
\end{cor}

%% Another direct application is the incompressibility of the $p$-multiplication map $A\to A$ (i.e. its essential dimension equals $\dim A$) (Corollary \ref{cor:mult-by-p incompressible}). This confirms the incompressibility part of a conjecture of Brosnan (see e.g. \cite{FS-incompressibility}*{Conjecture 6.1}). When $p>\dim A$ and the abelian variety $A$ has good unramified reduction at $p$, Brosnan's conjecture follows from \cite{FKW-ess-dim-prismatic} via a prismatic cohomological argument. We should note that \cites{FKW-ess-dim-prismatic,FS-incompressibility} also prove $p$-incompressibility\footnote{This means the map stays incompressible after base change along any finite morphism whose degree is prime to $p$.} in the cases they consider. The argument we give here is quite different from theirs.

The starting point of our proof is the following  observation: if the isogeny $\alpha$ is birationally obtained from $Y\to Y'$ via base change, then the Galois group $G\cong \ker(\alpha)$ acts birationally and faithfully on  $Y$, which can be chosen as the image of the abelian variety $A$ by  a rational map with connected fibers.

In Section~\ref{sec.2} we show that such 
rational images of abelian varieties are quite special.
The albanese map  $Y\dashrightarrow \alb(Y)$ is surjective and its general fibers are 
rationally connected.

The induced map $A\dashrightarrow Y\dashrightarrow \alb(Y)$ is a morphism of abelian varieties; its kernel gives the abelian subvariety $B$ in Theorem~\ref{thm:ess.dim.}.
The remaining task is to understand the $(G\cap B)$-action
on the fibers of $Y\dashrightarrow \alb(Y)$.

In Section~\ref{sec.3} we show that an
abelian $p$-group of rank $>\frac{p}{p-1}n$ can not act faithfully on a 
rationally connected variety of dimension $n$. This generalizes some earlier work of Prokhorov \cite{Pro-p-group-Cr3}.

Putting these together, we obtain the proofs of
Theorem~\ref{thm:ess.dim.} and Corollary~\ref{cor:formula} in  Section~\ref{sec.4}. We also raise some  questions about the essential dimension of endomorphisms of other varieties.

Other applications of our results are discussed in Section~\ref{sec.5}.

We remark that the study of essential dimension is often related to questions around the Cremona groups. Dolgachev asked \cite{dolg-lect} whether the essential dimension $\ed(G)$ of a finite group $G$ is always bounded from below by the smallest integer $n$ such that $G$ embeds into the Cremona group $\mathrm{Cr}(n)$. This is open for non-abelian groups.

%% quotient (Paragraph \ref{defn:MRC}) of $Y$ is again birational to an abelian variety (Proposition \ref{prop:MRC quotient abelian}). Equivalently, the albanese morphism of $Y$ has rationally connected general fibers. The next observation is that birational automorphism groups of rationally connected varieties cannot contain abelian $p$-groups of large ranks\footnote{It is well known that the essential dimension $\ed(G)$ of a finite group $G$ is related to Cremona groups containing $G$. See e.g. \cite{Ser-cremona}*{Th\'eor\`eme 3.6} which determines $\ed(A_6)$ using the classification of finite subgroups of the plane Cremona group.}. More precisely, the rank is at most $\frac{p}{p-1}$ times the dimension (Corollary \ref{cor:abelian action on RC}). This translates to the essential dimension bound in Theorem \ref{thm:ess.dim.}. 

\subsection*{Acknowledgement}

We thank Patrick Brosnan, Igor Dolgachev, Louis Esser, Najmuddin Fakhruddin, Olivier Haution, Mark Kisin, Shizhang Li, Joaqu\'in Moraga, Zsolt Patakfalvi, Angelo Vistoli and Jesse Wolfson for many helpful discussions and comments, and Dan Edidin for pointing us to \cite{MR1466694}. We also thank the anonymous referee for careful reading and suggestions. 
Partial financial support to JK was provided by  the NSF under grant number DMS-1901855, while ZZ was partially supported by the NSF Grants DMS-2240926, DMS-2234736, a Clay research fellowship, as well as a Sloan fellowship.

\section{Rational images of abelian varieties}\label{sec.2}

% \begin{defn}
% An algebraic fibre space is a surjective projective morphism $f\colon X\to Y$ between normal varieties with connected fibers.
% \end{defn}

Recall that a variety $X$ is called {\it uniruled} (resp. {\it rationally connected}) if through any point (resp. any pair of points) on $X$ there is a rational curve. See \cite{Kol-rational-curve} for the general theory of such varieties.

 As we indicated, the first step to the proof of Theorem \ref{thm:ess.dim.} is the following description of the rational images of abelian varieties.

%  \begin{prop} \label{prop:MRC quotient abelian}
% Let  $A$ be  an abelian variety and $f\colon A\dashrightarrow Y$  a 
% dominant, rational map with connected general fibers. Assume that  $Y$ is 
% not uniruled. Then there is an  abelian subvariety  $B\subset A$ such that 
% $f$ is birational to the quotient map  $A\to A/B$. In particular, $Y$ is 
% birational to the abelian variety $A/B$.
% \end{prop}

% \begin{proof}  Let $\Gamma\subset A\times Y$ be the closure of the graph 
% of $f$
%    and $\pi_A, \pi_Y$ the coordinate projections. Let  $E\subset \Gamma$ be 
% the exceptional divisor of $\pi_A|_{\Gamma}:\Gamma\to A$.
%    By \cite{Kol-rational-curve}*{Theorem VI.1.9}  every irreducible 
% component of
%    $\pi_Y(E)$ is uniruled, hence $\pi_Y(E)\neq Y$.

%    Hence the restriction of $f$ to $A\setminus \pi_A(E)$ is a morphism with 
% proper general fibers $B_y$. Since $B_y$ is a general fiber, its normal 
% bundle is trivial, and so is its canonical class by the adjunction 
% formula.
%    Thus $B_y$ is the translate of  an  abelian subvariety   $B\subset A$ by
%    \cite{Ueno-book}*{Corollary 10.5}, which is independent of $y$ since 
% abelian subvarieties have no flat deformations. \end{proof}

\begin{prop} \label{prop:MRC quotient abelian}
Let $f\colon X\dashrightarrow Y$ be a dominant, rational map with connected general fibers. Assume that  $X$ is  an abelian variety and $Y$ is not uniruled. Then $Y$ is  birational to an abelian variety.
\end{prop}

We give several different proofs of this proposition.

\begin{proof}[First proof of Proposition \ref{prop:MRC quotient abelian}]
Let $\Gamma\subset X\times Y$ be the closure of the graph
of $f$
   and $\pi_X, \pi_Y$ the coordinate projections. Let  $E\subset \Gamma$ be
the exceptional divisor of $\pi_X|_{\Gamma}:\Gamma\to X$.
   By \cite{Kol-rational-curve}*{Theorem VI.1.9}  every irreducible
component of
   $\pi_Y(E)$ is uniruled, hence $\pi_Y(E)\neq Y$.

Hence the restriction of $f$ to $X\setminus \pi_X(E)$ is a morphism with
proper general fibers $B_y$. Since $B_y$ is a general fiber, its normal
bundle is trivial, and so is its canonical class by the adjunction
formula.
   Thus $B_y$ is the translate of  an  abelian subvariety   $B\subset X$ by
   \cite{Ueno-book}*{Corollary 10.5}, which is independent of $y$ since
abelian subvarieties have no flat deformations. It follows that $Y$ is birational to $X/B$ which is an abelian variety.
\end{proof}

\begin{proof}[Second proof of Proposition \ref{prop:MRC quotient abelian}] 
By \cite{KL-quotient-of-CY}*{Theorem 14},
there is a finite, \'etale cover  $h:X'\to X$ and  a product decomposition
$X'\cong F'\times Y'$ such that the composite
$f\circ h\colon X'\to X\dashrightarrow Y$ factors as
$$
X'\stackrel{\pi}{\longrightarrow} Y' \stackrel{\tau}{\dashrightarrow} Y,
$$
where $\pi$ is the coordinate projection, and $\tau$ is generically finite.
Note that $X'$, being a finite \'etale cover of an abelian variety, is also an abelian variety. The factor $F'$ is a subvariety with trivial canonical class, hence is also an abelian varieties by \cite{Ueno-book}*{Corollary 10.5}, and so
 $h\bigl(F'\times \{y'\}\bigr)\subset X$ is a translate of an
abelian subvariety $B_{y'}\subset X$ for general $y'\in Y'$.
Here  $B_{y'}$ depends continuously on $y'$ for general $y'\in Y'$, hence
in fact $B:=B_{y'}$ is independent of $y'$. 
Thus  $f$ factors as
$$
f\colon  X\to X/B \dashrightarrow Y,
$$
where $X/B$ is an abelian variety, and $ X/B \dashrightarrow Y$ is
generically finite. We assumed that $f$ has connected general fibers,
so $ X/B \dashrightarrow Y$ is birational.
\end{proof}

\begin{rem}[Communicated with Zsolt Patakfalvi]
Proposition \ref{prop:MRC quotient abelian} also holds in positive characteristics. As in the first proof above, the map $f\colon X\dashrightarrow Y$ has proper general fibers over an open set $U\subseteq Y$. Let $\mathrm{Frob}^e\colon U' \to U$ be an iterated Frobenius morphism for some $e\gg 0$ and let $X'$ to be the normalization of the reduced subscheme of $X \times_Y U'$. Then the generic fiber $F$ of $X'\to U'$ is geometrically reduced. Let $g\colon F\to X$ be the induced
morphism. By \cite[Theorem 1.1]{Tanaka} one has $K_F = g^* K_X - E$ for some effective divisor $E$, hence $-K_F$ is pseudo-effective (equivalently, $\omega^*_F$ is weakly-positive as in \cite[Definition 2.2]{EP-albanese-char-p}). Note that $F$ has maximal albanese dimension since it admits a generically finite map to the abelian variety $X$. By \cite[Proposition 3.2]{EP-albanese-char-p}, this implies that the general fiber of $X'\to U'$ are abelian varieties and thus $K_{F}=E=0$. By \cite[Theorem 1.1]{Tanaka}, we deduce that the generic fiber of $f$ is also geometrically reduced and hence the general fibers of $f$ are abelian varieties as well. We then conclude as in the above proofs of Proposition \ref{prop:MRC quotient abelian}.
\end{rem}

\begin{rem}
Having in mind possible generalizations, it may be worthwhile to recall how
Proposition~\ref{prop:MRC quotient abelian} relates to
various conjectures of Iitaka; see \cite{Ueno-book} for an introduction. 

Let us denote the general fiber of $X\dashrightarrow Y$ by $F$.  Iitaka's so called $C_{n,m}$  conjecture  says that $\kappa(X)\ge \kappa(F)+\kappa(Y)$, where $\kappa(X)$ is the Kodaira dimension of $X$ which takes values in $\{-\infty\}\cup \bN$.  On the other hand, the non-vanishing conjecture predicts that if $Y$ is not uniruled then $\kappa(Y)\ge 0$. Combining these two, and since $\kappa(X)=0$, we see that $\kappa(F)=\kappa(Y)=0$. This implies that $F$ is the translate of an abelian subvariety and $Y$ is birational to the quotient abelian variety as before.
\end{rem}

\begin{cor}\label{prop:MRC quotient abelian.cor}
  Let $X$ be an abelian variety and
   $f\colon X\dashrightarrow Y$  a dominant, rational map with connected general fibers. Then 
  the albanese map  $Y\to \alb(Y)$ is surjective and its general fibers are 
  rationally connected.
\end{cor}

\begin{proof}
  By \cite{Kol-rational-curve}*{Theorems 5.4 and 5.5} and \cite{GHS-RC-fibration-section}*{Corollary 1.4}, for any variety $Y$
  there is a dominant rational map  $\sigma\colon Y\dashrightarrow Z$ such that $Z$ is not uniruled, and whose general fibers are rationally connected.
  This map is called the {\it  maximal rationally connected} or {\it MRC} fibration, and it is birationally unique.

  If $Y$ is the rational image of an abelian variety, then so is $Z$, hence it is birational to an 
  abelian variety by Proposition~\ref{prop:MRC quotient abelian}.
  Thus the the albanese map $Y\to \alb(Y)$ is the  MRC fibration of $Y$.
  \end{proof}

\section{Actions of abelian $p$-groups}\label{sec.3}

Another ingredient in the proof of Theorem \ref{thm:ess.dim.} is an optimal bound on the ranks of abelian $p$-subgroups of the birational automorphism groups of rationally connected varieties. For $p>\dim +1$, \cite{j-xu} obtains  optimal bounds using the fixed point theorems of \cite{MR3905117}.
We observe that,  when there are no fixed points, the method of \cite{MR3905117} still gives small orbits.
This leads to  optimal bounds for $p\leq \dim+1$.

We first prove some general results about abelian   $p$-groups acting on varieties. For a prime $p$ and an integer $m$, set $\nu_p m = \sup\{a\in\bN \,:\, p^a | m \}\in \bN\cup \{+\infty\}$.

\begin{prop}\label{prop.1}
  Let $K$ be an algebraically closed field, and $G$  an abelian  $p$-group acting on a  smooth, projective $K$-scheme $X$ of pure dimension $n$.
Assume that $\chr K\neq p$. %and   write $\chi(X, \o_X)=p^ap'$ where $p\nmid p'$.
  Then there is a subgroup $H\le G$ such that
  $H$ has a fixed point on $X$, and
  $$
  \log_p[G:H]\leq \nu_p (\chi(X,\cO_X))+\tfrac{n}{p-1}.
  $$
\end{prop}

\begin{proof}
Let $a=\nu_p (\chi(X,\cO_X))$ which we may assume is finite.
Let $p^c$ be the  smallest size of a $G$-orbit.
Then $p^c$ divides the  size of every $G$-orbit.

As we discuss in Paragraph~\ref{eit.s}, the  Chern numbers of $X$ are represented by $G$-invariant $0$-cycles, hence they are divisible by $p^c$.

By \cite[Lemma 1.7.3]{hirz}, the exponent of the largest $p$-power dividing the denominator of the $n$th Todd class is
$\lfloor\tfrac{n}{p-1}\rfloor$, thus the exponent of the largest $p$-power dividing   $\operatorname{td}_n(X)$
is at least $c-\tfrac{n}{p-1}$.
By the Hirzebruch-Riemann-Roch theorem,  $\chi(X, \o_X)=\operatorname{td}_n(X)$, so this exponent equals $a$ by assumption.
Thus $c\leq a+\tfrac{n}{p-1}$. Now take $H$ to be the stabilizer of a point in a smallest orbit. 
\end{proof}

% We get the following corollary.

\begin{cor}\label{rank.cor}
  Let $K$ be a field of characteristic $\neq p$, and $G$  an abelian  $p$-group acting faithfully on a  smooth, projective
$K$-scheme $X$ of pure dimension $n$. % Write $\chi(X, \o_X)=p^ap'$ where $p\nmid p'$. 
Then
  $$
  \rank (G) \leq \nu_p (\chi(X,\cO_X))+\tfrac{p}{p-1}n. % \hfill\qed
  $$
\end{cor}

\begin{proof}
Let $H\le G$ be the subgroup obtained from Proposition \ref{prop.1} and let $x\in X$ be an $H$-fixed point. If $\chr K\neq p$ then
$H$ acts faithfully on the tangent space $T_xX$, so $H$ is isomorphic to a
subgroup of $\GL_n(K)$. Since $H$ is abelian,  it is diagonalizable, so it is the product of at most $n$ cyclic groups. The statement now follows as $\rank(G)\le \rank(H)+\log_p [G:H]$.
\end{proof}

\begin{rem}  
In fact, we get a stronger result, that 
$G\cong G_1\times G_2$, where $\rank (G_1)\leq n+\nu_p (\chi(X,\cO_X))$, and
 $|G_2|\leq p^{n/(p-1)}$.
\end{rem}

If $X$ is  smooth, projective and  rationally connected over a field of characteristic $0$, then  $\chi(X, \o_X)=1$. So we get the following special case.

\begin{cor} \label{cor:abelian action on RC}
  Let $G$ be an abelian  $p$-group acting faithfully on a  smooth, projective, rationally connected variety $X$ of dimension $n$ over a field of characteristic $0$. Then
  $$
  \rank (G) \leq \tfrac{p}{p-1}n \leq 2n.  \hfill\qed
  $$
\end{cor}

\begin{expl} \label{xe.2}  Let $\epsilon$ be a primitive  $m$th root of unity.
  The dihedral group $D_{2m}$ 
 acts on $\p^1$ by  $x\mapsto  \epsilon x, x\mapsto x^{-1}$.
  Thus $D_{2m}^n$ acts on $(\p^1)^n$ coordinate wise.

  For $m=2$ we get a $(\z/2)^{2n}$ action on  $(\p^1)^n$.
  So the rank bound is sharp in Corollary~\ref{cor:abelian action on RC} for $p=2$.

  Esser pointed out that 
  $(\z/p)^{p}$ acts diagonally on $
  X_p:=\bigl(\tsum_i x_i^p=0\bigr)\subset\p^{p}.
  $
  Powers of $X_p$ show that 
the bound is sharp for every $p$.
\end{expl}

\begin{say}[Equivariant intersection theory]\label{eit.s}
  Let $G$ be a reductive, abelian group acting on a smooth, projective variety $X$. In other words, $G$ is abelian, its identity component $G_0$ is a torus $\bG_m^r$ (for some $r\in\bN$) and $|G/G_0|$ is not divisible by the characteristic of the field. Then one can do intersection theory with $G$-invariant cycles.
  That is, the intersection of $G$-invariant cycles can be
  represented by $G$-invariant cycles. This is obvious if the intersection has the expected dimension, but not otherwise. Example~\ref{xe.3} shows that this in fact fails for non-abelian 2-groups.

  For  $G=T$  a torus, $T$-equivariant intersection theory is treated in 
  \cite{MR2689936}. A more sophisticated version is in \cite{MR1466694}.
  Note, however, that torus actions have additional features, so not all claims in these papers generalize.
  The  treatment in \cite{MR3905117}*{Section 4} covers both
  reductive, abelian groups and non-abelian $p$-groups in characteristic p. 

  Here  we discuss the assertion that we actually use.

  %% (Defining the intersection product using the normal cone as in  \cite[Chap.6]{Fulton84} shows that Lemma~\ref{eit.s}.2 in fact implies the general statement.)

  \medskip
      {\it Claim~\ref{eit.s}.1.}  Let $G$ be a reductive, abelian group acting on a smooth projective variety $X$. Then  the Chern numbers of $X$ are represented by $G$-invariant $0$-cycles.

      \medskip

      {\it Proof.} Using the formulas in  \cite[Sec.3.2]{Fulton84}, this is equivalent to representing the Segre numbers.
      Let $\pi:P\to X$ be the projectivization of $T_X$ with tautological line bundle $\o_P(1)$.  Set $D:=c_1\bigl(\o_P(1)\bigr)$.

      By  \cite[Sec.3.1]{Fulton84},  the Segre  classes of $X$ are $s_i(X):=\pi_*\bigl(D^{n-1+i}\bigr)$, and the Segre numbers are the  intersection products  $s_{i_1}(X)\cdots s_{i_r}(X)$ for
$i_1+\cdots+i_r=n$. Repeatedly using the projection formula, the Claim is implied by  the following.

        \medskip
            {\it Lemma~\ref{eit.s}.2.}  Let $Z\subset P$ be a $G$-invariant, $G$-irreducible cycle.  Then there is a  $G$-invariant divisor $D_Z\sim D$ such that    $Z\not\subset \Supp D_Z$.

      \medskip

{\it Proof.} The $G$-action on $X$ lifts to a $G$-action on $P$. Pick $z\in Z$ and choose some $G$-linearized very ample line bundles $L,M$ on $P$ such that $D\sim L-M$. Since $G$ is reductive and abelian, both $H^0(L)$ and $H^0(M)$ are spanned by $G$-equivariant sections. Thus there exist $G$-equivariant sections $s_L\in H^0(L)$ and $s_M\in H^0(M)$ such that $z$ is not in their zero locus. Set $D_Z:=(s_L=0)-(s_M=0)$. \qed
% The $G$-action on $X$ lifts to a $G$-action on $P$ and a $G$-linearization of $\o_P(1)$.
% Thus $G$ acts on the vector space $V$ of rational sections of $\o_P(1)$.
% Since $G$ is reductive and abelian, $V$  is spanned by $G$-equivariant rational sections.
% Pick $z\in Z$. Then there is a  $G$-equivariant rational section
% $s_z$ such that $s_z(z)\not\in\{ 0, \infty\}$. Set $D_Z:=(s_Z)$. \qed
\end{say}

\begin{expl}\label{xe.3} Let $G$ be the  group of order 32 acting on $\p^1\times \p^1$, generated by the $(\z/2)^4$-action as in   Example~\ref{xe.2}, plus the interchanging the 2 copies of $\p^1$.  The points  $(0,0), (1,1), (i,i)$ each have a $G$-orbit of size 4, all other $G$ orbits have size $\geq 8$.

Let $S\to \p^1\times \p^1$ be the blow-up of the 3 orbits of size 4.
The $G$-action lifts to $S$, and   all $G$-orbits on $S$ have size $\geq 8$.
However  $c_1(S)^2=8-12$ is not divisible by 8.
\end{expl}

\section{Essential dimension of isogenies}\label{sec.4}

We now turn to the proof of our main results.

%% Again, we assume that the base field is algebraically closed of characteristic zero. First let us recall the notion of MRC quotient which will be used in the proof. 

%% \begin{say} \label{defn:MRC}
%% By \cite{Kol-rational-curve}*{Theorems 5.4 and 5.5} and \cite{GHS-RC-fibration-section}*{Corollary 1.4}, any variety $X$ has a maximal rationally connected (MRC) fibration $X\dashrightarrow Y$ whose fibers are rationally connected while the base is not uniruled. This MRC fibration is  unique up to birational maps of $X$ and $Y$. The base of the MRC fibration is called the MRC quotient of $X$.
%% \end{say}

\begin{proof}[Proof of Theorem \ref{thm:ess.dim.}]
By the definition of essential dimension, there exists a Cartesian diagram 
\[
\xymatrix{
A \ar[r] \ar@{-->}[d] & A' \ar@{-->}[d] \\
Y \ar[r] & Y'
}
\]
where the horizontal maps are finite and $\dim Y = \dim Y' = \ed(\alpha)$. Replacing $Y'$ by its normalization in $A'$ and $Y$ by the corresponding fiber product, we may assume that the rational maps $A'\dashrightarrow Y'$ and $A\dashrightarrow Y$ have connected fibers. In particular, the function field $\bC(Y)$ of $Y$ is algebraically closed in $\bC(A)$. From this we see that the Galois action of $G:=\ker(\alpha)$ on $\bC(A)$ descends to $\bC(Y)$: in fact, any element in the Galois conjugate of $\bC(Y)$ is finite over $\bC(Y)$, hence is contained in $\bC(Y)$ since $\bC(Y)$ is algebraically closed in $\bC(A)$. It follows that $G$ acts birationally and faithfully on $Y$.

Let $Y\to \alb(Y)$ be the  albanese map and  $F\subset X$ a general fiber.
The composite  $A\dashrightarrow Y \to \alb(Y)$ is a morphism of abelian varieties; 
let $B\subseteq A$ be its kernel. It is connected by construction.

Note that $Y\to \alb(Y)$ is $G$-equivariant, and
 $L:=G\cap B$ acts  trivially on $\alb(Y)$. Hence $L$ must act faithfully on $F$.

%% MRC fibration and let $F$ be its general fiber. The functoriality of the MRC fibration implies that $G$ also acts birationally on the MRC quotient $Z$ and the MRC fibration is $G$-equivariant. By Proposition \ref{prop:MRC quotient abelian}, up to birational modifications we may assume that $Z$ is an abelian variety. The rational map $A\dashrightarrow Z$ is then a morphism of abelian varieties. Let $B\subseteq A$ be its kernel which is an abelian subvariety (it is connected by construction). The induced (translation) action of $L=G\cap B$ on $Z$ is trivial, hence $L$ must act faithfully on $F$. Since

The fiber $F$ is rationally connected by  Corollary~\ref{prop:MRC quotient abelian.cor}.
Thus Corollary~\ref{cor:abelian action on RC} implies that  $\dim F\ge \frac{p-1}{p} \rank_p (L)$. Therefore,
\[
\ed(\alpha)=\dim Y = \dim A - \dim B + \dim F \ge \dim A - \dim B + \tfrac{p-1}{p} \rank_p (L),
\]
as desired.   
\end{proof}

In order to prove  Corollary~\ref{cor:formula} we need some  properties of essential dimension.

\begin{lem} \label{lem:ess.dim subadditive}
Let $f_i\colon X_i\to Y$ $(i=1,2)$ be finite morphisms and let $f\colon X=X_1\times_Y X_2\to Y$ be the fiber product. Then $\ed(f)\le \ed(f_1)+\ed(f_2)$.
\end{lem}

\begin{proof}
Clear from the definition.
\end{proof}

\begin{lem} \cite{BR-ess-dim}*{Theorem 6.1} \label{lem:ess.dim abelian cover}
Let $f\colon X\to Y$ be a finite abelian cover with Galois group $G$. Then $\ed(f)\le \rank(G)$.
\end{lem}

\begin{proof}
We may write $f$ as the fiber product of $\rank(G)$ cyclic covers. Since any cyclic cover has essential dimension one, the statement follows from Lemma~\ref{lem:ess.dim subadditive}.
\end{proof}

\begin{proof}[Proof of Corollary \ref{cor:formula}]
We first prove that
\[
\ed(\alpha)\le \dim A -\dim B + \rank(\ker(\alpha)\cap B)
\]
for any abelian subvariety $B$ of $A$. Let $L=\ker(\alpha)\cap B$ and let $K=\ker(\alpha)/L$. Then the isogeny $A\to A'$ factors through $A_1=A/L$. Let $r=\rank(L)$. We can find another isogeny $A_2\to A'$ with kernel rank $r$ such that the fiber product $A_1\times_{A'} A_2$ factors through $A$ (it suffices to lift $\pi_1(A_1)\to L$ to some $\pi_1(A')\to L'$ with $\rank(L')=r$). By Lemma \ref{lem:ess.dim abelian cover}, we have $\ed(A_2/A')\le r$, hence by By Lemma \ref{lem:ess.dim subadditive} we only need to show that 
\[\ed(A_1/A')\le \dim A -\dim B.\] 
To see this, let $A_0:=A/B$ and let $A'_0:=A_0/K\cong A'/\alpha(B)$. 
\[
\xymatrix{
A_1=A/L \ar[r] \ar[d] & A' \ar[d] \\
A_0 = A/B \ar[r] & A'_0 \\
}
\]
Then the natural map $A_1\to A'\times_{A'_0} A_0$ is an isomorphism: in fact, its kernel is the same as the kernel of $K\cong\ker(A_1\to A')\to A_0$, but the latter map is injective by construction. It follows that $\ed(A_1/A')\le \dim A_0=\dim A -\dim B$ and this gives the desired inequality.

For the reverse direction, we may assume that $\ed(\alpha)<\dim A$ (otherwise there is nothing to prove). By Theorem \ref{thm:ess.dim.}, there exists some abelian subvariety $B$ of $A$ such that
\[
\ed(\alpha)\ge \dim A -\dim B + \tfrac{p-1}{p}\rank_p (L)
\]
for every prime $p$, where $L=\ker(\alpha)\cap B$ as before. In particular, $ \frac{p-1}{p}\rank_p (L) \le \dim B-1$ (since $\ed(\alpha)<\dim A$). Since $\deg(\alpha)=|L|$ is coprime to $(\dim A)!$, there exists some prime $p>\dim A\ge \dim B$ such that $\rank_p (L) = \rank (L)$. For this $p$ we get
\[
\rank (L) = \rank_p (L)\le \tfrac{p}{p-1}(\dim B-1) < \dim B<p,
\]
hence $ \lceil\frac{p-1}{p}\rank_p (L) \rceil = \rank (L)$. This finishes the proof.
\end{proof}

\begin{cor}
Let $\alpha\colon A\to A'$ be an isogeny of simple abelian varieties whose degree is relatively prime to $(\dim A)!$. Then
$
\ed(\alpha) = \min\{\dim A, \rank(\ker(\alpha))\}.
$
\end{cor}

\begin{proof}
Since $A$ is simple, the only abelian subvarieties are the trivial ones, hence the statement is clear from Theorem \ref{thm:ess.dim.}.
\end{proof}

We conclude with some questions and remarks.  \cite{FKW-ess-dim-prismatic}*{Prop.2.3.10}  proves a more general result about the essential dimensions of mod $p$ homology covers. It would be interesting to find a geometric version of their result. A natural geometric replacement of the mod $p$ homology cover is the \'etale cover pulled back from the $p$-multiplication map on the albanese variety, and one may guess (as a geometric analog of \emph{loc. cit.}) that the essential dimension of this cover is at least the albanese dimension of the variety. However, the following example shows that this is not always true.

\begin{expl}
Let $A$ be an abelian variety of dimension $n$ and let $C$ be a smooth projective curve with a faithful $G:=(\bZ/p)^{2n}$-action such that $C/G\cong \bP^1$. By identifying $G$ with $A[p]$ (which acts on $A$ via translation), we get a diagonal $G$-action on $A\times C$ without fixed point. Let $X:=(A\times C)/G$ and consider the (\'etale) cover $f\colon A\times C\to X$. We have a Cartesian diagram
\[
\xymatrix{
A\times C \ar[r] \ar[d] & X \ar[d] \\
A \ar[r] & A/G
}
\]
where the vertical arrows are induced by the natural projections and the bottom row can be identified with the $p$-multiplication map on $A\cong A/G$. We claim that the map $X\to A':=A/G$ is the albanese morphism of $X$. For this it suffices to show that the induced map $H_1(X,\bZ)_{\tf}\to H_1(A',\bZ)$ between the torsion free part of the integral homology is an isomorphism. 

Observe that the projection $X\to A'$ is a smooth fibration with (connected) fiber $C$, hence the induced map $\pi_1(X)\to \pi_1(A')$ is surjective by the homotopy exact sequence. This implies that the map $H_1(X,\bZ)\to H_1(A',\bZ)$ is surjective as well. It remains to show that $b_1(X)=b_1(A')=2n$. By Hodge theory, this is also equivalent to $h^{1,0}(X)=n$, or $\dim H^0(\Omega^1_{A\times C})^G = n$. This can now be checked with a direct calculation, noting that by K\"unneth formula $H^0(\Omega^1_{A\times C})\cong H^0(\Omega^1_A)\oplus H^0(\Omega^1_C)$, the $G$-action is trivial on the first factor, while $H^0(\Omega^1_C)^G =H^0(\Omega_{C/G}^1)=H^0(\Omega_{\bP^1}^1)=0$. Thus we see that $X\to A'$ is the albanese morphism and the albanese dimension of $X$ is equal to $n=\dim A'$. The cover $f\colon A\times C\to X$  is therefore the pullback of the $p$-multiplication map on the albanese variety by the previous diagram. We also have another Cartesian diagram
\[
\xymatrix{
A\times C \ar[r] \ar[d] & X \ar[d] \\
C \ar[r] & C/G
}
\]
similar to the one we have before. It gives $\ed(f)\le \dim C =1$ and in particular the essential dimension $\ed(f)$ can be smaller than the albanese dimension. 
\end{expl}

Recall that a dominant morphism $f\colon X\to X$ of normal, projective  varieties is called a {\it polarized} (resp. {\it int-amplified}) endomorphism if there exists some ample line bundle $L$ on $X$ such that $f^*L=qL$ some integer $q>1$ (resp. $f^*L-L$ is ample). Clearly, every polarized endomorphism is int-amplified. The simplest example is the multiplication-by-$p$ map on abelian varieties. Similarly, on a  normal toric variety,  the  $p$th power  map extends from the open torus to a finite morphism of the toric variety.  There are many other examples such as finite self-maps of $\bP^n$ (or of any variety of Picard number $1$). In positive characteristics there are also the Frobenius maps. We proved that the multiplication maps on abelian varieties are incompressible. The $p$th power maps on toric varieties are even  $p$-incompressible for all $p$, see  \cite{RY-ess-dim-fix-point-method}*{Theorem 1.2} or \cite{BF-incompressibility}*{Proposition 10}. Motivated by these results, it seems natural to ask:

\begin{que}
Is every polarized endomorphism incompressible?
\end{que}

Fakhruddin pointed out that int-amplified endormorphisms in characteristic zero and polarized endomorphisms in positive characteristic can have small essential dimension. To get an int-amplified example, choose pairwise coprime integers $a_1,\cdots,a_n>1$, let $X = (\bP^1)^n$ and consider the endomorphism $f\colon X\to X$ 
given coordinate-wise by the map $x \mapsto x^{a_i}$ on the $i$-th factor. Then $f$ is int-amplified but its essential dimension is one since its Galois group is cyclic. To get a polarized example in characteristic $p$, consider the endomorphism $g\colon X\to X$ given by the Artin–Schreier cover $x \mapsto x^p-x$ on each product factor. The Galois group is $(\bZ/p)^n$ which has essential dimension one by \cite{Led-ess-dim-1}. By \cite{BR-ess-dim}*{Theorem 3.1(c)}, this implies that $\ed(g)=1$. For the same reason, any isogeny of abelian varieties with kernel $(\bZ/p)^n$ in characteristic $p$ has essential dimension one; therefore, the statement of Corollary \ref{cor:formula} does not hold in positive characteristic in general.

One limitation of our current approach is that it does not give much information about essential dimensions of morphisms between varieties of general type. It would be interesting to extend our geometric arguments to this more general setting. Such an extension may also be useful in addressing the $p$-incompressibility part of Brosnan's conjecture.

\section{Other results on group actions}\label{sec.5}

The essential dimension $\ed(G)$ of a finite group $G$ is related to Cremona groups containing $G$. For example, \cite{Ser-cremona}*{Th\'eor\`eme 3.6}  determines $\ed(A_6)$ using the classification of finite subgroups of the plane Cremona group.

More generally, one can ask which groups can act on
$n$-dimensional rationally connected  varieties.
For $S_m$ and $A_m$ actions, 
\cite{esser2023symmetries} uses  $p$-subgroups for some
$p>n+1$ to prove that  $m\leq n^2+Cn$.
\cite{esser2023symmetries} also gives examples where $m>n+\sqrt{2n}$.

Using Corollary~\ref{cor:abelian action on RC}, we can use  subgroups  $(\z/2)^{2n+1}\subset S_{4n+2}$ and $(\z/2)^{2n+1}\subset A_{4n+4}$, which give  the following   bounds.

\begin{cor} \label{cor.4.1}
  Assume that  $S_m$ (resp.\ $A_m$)   acts faithfully on a  smooth, projective, rationally connected variety $X$ of dimension $n$ over a field of characteristic $0$. Then  $m\leq 4n+1$  (resp.\ $m\leq 4n+3$).
%%Thus $\ed(S_n)\geq \frac14(n-1)$ and $\ed(A_n)\geq \frac14(n-3)$.
  \qed
  \end{cor}

For local rings we get the following variant of the results in Section \ref{sec.3}.

\begin{cor} 
Let $R$ be the local ring $($or the completion of such$)$ of a closed point on a variety of dimension $n$ with rational singularities over a field of characteristic 0, and $G$ an abelian  $p$-group acting on $R$.
 % Assume that $R$ has rational singularities.
Then there is a subgroup $H\le G$ such that
  $H$ has a fixed point on every $G$-equivariant resolution of $\spec R$, and
  $$
  \log_p [G:H]\leq \tfrac{n-1}{p-1}.
  $$
  Therefore
  $$
  \rank (G)\leq n+\tfrac{n-1}{p-1}\leq 2n-1.
  $$
\end{cor}

\begin{proof}
Let $Y\to \spec R$ be a $G$-equivariant resolution such that
the reduced preimage of the closed point is a simple normal crossing divisor $E\subset Y$.  Note that $\chi(E, \o_E)=1$ by \cite{steenbrink} since 
$R$ has rational singularities.

Let $\{E_i:i\in I\}$ be the irreducible components of $E$. For $1\leq m\leq n$ set
$$
E^{m}:=\textstyle{\amalg}_{J\subset I, |J|=m} \cap_{i\in J} E_i,
$$
with projection $\pi^m:E^m\to E$. 
There is an exact sequence
$$
0\to \o_E\to \pi^1_*\o_{E^1}\to \pi^2_*\o_{E^2}\to \cdots \to \pi^n_*\o_{E^n}\to 0,
$$
therefore
$$
1=\chi(E, \o_E)=\tsum_{m=1}^n (-1)^{m-1} \chi(E^m, \o_{E^m}).
$$
Thus at least one of the $\chi(E^m, \o_{E^m}) $ is not divisible by $p$.
We can apply Proposition~\ref{prop.1} to that $E^m$ to get $H\le G$ with an $H$-fixed point on $E^m$.  Its image by $\pi^m$ is an $H$-fixed point on $Y$.
\end{proof}

\begin{rem} \label{CY.rem}
  Let $G$ be an abelian  $p$-group acting faithfully on a  smooth, projective variety $X$ of dimension $n$ over a field of characteristic $0$.
Assume that $\chi(X, \o_X)= \pm 1$ or $\pm 2$. Then Corollary~\ref{rank.cor} implies that 
\[
\rank G \leq
\left\{
\begin{array}{l} \tfrac{p}{p-1}n \qtq{if $p$ is odd, and}\\[1ex]
  2n+1 \qtq{if $p=2$.}
\end{array}
\right.
\eqno{(\ref{CY.rem}.1)}
\]
This covers many even dimensional Calabi-Yau varieties.
However, if $X$ is a 3-dimensional Calabi-Yau variety, then
$\chi(X, \o_X)=0$, and our methods give no bound.

Nonetheless, Moraga conjectures that the bounds (\ref{CY.rem}.1) hold in odd dimensions as well.

Note that the bounds (\ref{CY.rem}.1) are optimal in certain cases. Dolgachev pointed out that the abelian $2$-group of rank $5$ acts on the K3 surface given by the intersection of three diagonal quadrics in $\bP^5$.

% The automorphism groups of K3  surfaces are classified in  \cite{MR0958597}.
% The largest abelian 2-groups have rank 4.
\end{rem}

In another direction,  \cite{MR4069653} shows that an
abelian $p$-group action can not have a unique fixed point.
For these purposes,  \cite{MR4069653}
develops  a generalization of the Rost invariant \cite{rost-deg}. Here we give an elementary proof.

\begin{prop} Let $G$ be a reductive, abelian group acting on a proper variety $X$ over an algebraically closed field $K$.
  Assume that the component group $G/G_0$ is a $p$-group, $\mathrm{char}(K)\neq p$, $G$ has a smooth, isolated fixed point and $\dim X\geq 1$. Then $G$ has at least 1 other fixed point.
\end{prop}

Proof.  Assume first that $\dim X=1$. 
After normalization, we may assume that $X$ is smooth.
Let $x\in X$ be the fixed point.
Let $L$ be a very ample, $G$-linearized line bundle on $X$.
Then
$$
H^0(X, L^p)\onto  H^0(X, L^p\otimes \o_X/m_x^2) \qtq{is surjective.}
$$
Since $G$ is reductive, any element in $m_x /m_x^2$
 lifts  to a $G$-equivariant section
 $s\in H^0(X, L^p)$.
 Then $D:=(s=0)$ is a 0-dimensional $G$-scheme whose degree is divisible by $p$,
 and $\{x\}\subset D$ is an irreducible component of degree 1.
 Thus $D$ has another irreducible component that is fixed by $G$ and
 whose degree is not divisible by $p$.
 This gives another fixed point.

For $\dim X>1$ we use induction on the dimension.
We may assume that $X$ is projective.
Let $L$ be a very ample, $G$-linearized line bundle on $X$.
Then
$$
H^0(X, L)\onto  H^0(X, L\otimes \o_X/m_x^2) \qtq{is surjective.}
$$
Since $G$ is reductive and  abelian, it has an eigenvector on $m_x/m_x^2$,
which lifts  to a $G$-equivariant section
$s\in H^0(X, L)$. Thus $D:=(s=0)$ is $G$-invariant and
$x\in D$ is a smooth, isolated fixed point.
There is a unique irreducible component $x\in D'\subset D$.
By induction, $G$ has  at least 1 other fixed point on $D'$,  hence also on $X$. \qed

%%%%%%%%%%%%%%%%%%%%

% \begin{lem}[\cite{j-xu}*{Theorem 2.2}] \label{lem:p group in Bir(RC)}
% Let $X$ be a rationally connected variety of dimension $n$ and let $G\subseteq \Bir(X)$ be a finite $p$-subgroup for some $p>n+1$. Then $G$ is abelian and the rank of $G$ is at most $n$.
% \end{lem}

\bibliography{ref}

\end{document}